\documentclass{amsart}
\usepackage{amsthm,amssymb,amsfonts,setspace,verbatim,graphicx,mathtools,mathrsfs,commath,float}
\usepackage[hidelinks]{hyperref}
\usepackage{enumitem}
\usepackage{wrapfig}
\usepackage{ulem}
\newtheorem{ut}{Theorem}[section]
\newtheorem{uc}[ut]{Corollary}
\newtheorem{ul}[ut]{Lemma}

\newtheorem{uq}{Question}
\newtheorem{ucl}[ut]{Claim}
\theoremstyle{remark}
\newtheorem{ur}{Remark}[section]
\theoremstyle{definition}

\newtheorem{ue}[ut]{Example}

\DeclareMathOperator{\cl}{cl}
\DeclareMathOperator{\nc}{nc}

\title{Compactification of cut-point spaces}

\begin{document}
\author{David S. Lipham}
\address{Department of Mathematics \& Statistics, Auburn University at Montgomery, Montgomery 
AL 36117, United States of America}
\email{dsl0003@auburn.edu; dlipham@aum.edu}

\begin{abstract}
We show that if $X$ is a separable locally compact Hausdorff connected space with fewer than $\mathfrak c$ non-cut points, then $X$ embeds into a dendrite $D\subseteq \mathbb R ^2$, and the set of non-cut points of $X$  is a nowhere dense $G_\delta$-set. 
We then prove a Tychonoff cut-point space $X$ is weakly orderable if and only if $\beta X$ is an irreducible continuum. 
Finally, we show  every separable metrizable cut-point space densely embeds into a reducible continuum with no cut points. By contrast, there is a Tychonoff cut-point space each of whose compactifications has the same cut point.  The example raises some questions about persistent cut points in Tychonoff spaces.\end{abstract}
\maketitle

\section{Introduction}

Let $X$ be a connected topological space.  A point $x\in X$ is called a \textit{cut point} if $X\setminus \{x\}$ is disconnected.  If $X\setminus \{x\}$ has exactly two connected components, then $x$ is a \textit{strong cut point}. A \textit{cut-point space} (resp. \textit{strong cut-point space}) is a connected topological space in which every point is a cut point (resp. strong cut point).

In 1936, L.E. Ward  \cite{warr} famously proved: If $X$ is a  connected and  locally connected separable metrizable space in which every point is a strong cut point, then $X$ is homeomorphic to the real line.  In 1970, S.P. Franklin and G.V. Krishnarao strengthened Ward's result by showing  the word ``metrizable'' could be replaced with ``regular'' \cite{fra1}.  Then, in a short addendum  \cite{fra}   they claimed: \textit{If $X$ is a separable locally compact Hausdorff connected space in which every point is a strong cut point, then $X$ is homeomorphic to the real line.}  A mistake in the proof  was discovered in 1977 by A.E. Brouwer, who then re-proved the statement \cite[Theorem 8]{bro}.  More generally, Brouwer  showed  every separable locally compact Hausdorff cut-point space embeds into a dendrite \cite[Theorems 5 \& 7]{bro}. In Section 2 of this paper, we prove a stronger result using L.E. Ward's 1988  characterization of dendrons. 

\begin{ut}\label{t2} If   $X$ is a connected separable locally compact Hausdorff space with fewer than $\mathfrak c=|\mathbb R|$ non-cut points, then:
\begin{enumerate}
\item[\textnormal{(i)}] $X$ embeds into a dendrite whose cut points  are precisely the cut points of $X$; and 
\item[\textnormal{(ii)}] the set of non-cut points of $X$ is a nowhere dense $G_\delta$-set. 
\end{enumerate}\end{ut}
\noindent Every dendrite embeds into Wazewski's plane continuum \cite[10.37]{nad} (see Figure 1), so in fact the set $X$ in Theorem 1.1 embeds into a dendrite in the plane.  The result also implies that every separable Hausdorff continuum with only countably many non-cut points is a (plane) dendrite.









In Sections 3 and 4, we focus on non-dendritic compactifications of Tychonoff cut-point spaces, including weakly ordered spaces. 

A space $X$ is \textit{weakly orderable} if there exists a continuous linear ordering of the elements of $X$.  To be more precise, $X$ is weakly orderable if there is a continuous one-to-one mapping of $X$ into a Hausdorff arc.  Apparently,  every connected weakly ordered space is a strong cut-point space. 

In Section 3 we show that a connected Tychonoff space $X$ is weakly orderable if and only if $X$ is a cut-point space and $\beta X$ is an irreducible Hausdorff continuum (Corollary \ref{46}).  In this event, the Stone-\v{C}ech extension of the weak ordering epimorphism continuously orders the internal layers of $\beta X$.  We also show each connected weakly orderable normal space densely embeds into an irreducible Hausdorff continuum of the same weight  (Theorem \ref{47}).  This generalizes a result proved by Roman Duda in the separable metrizable setting \cite[Theorem 5]{dud}.

 We will see that each cut point of $X$ is a cut point of $\beta X$ (Theorem \ref{41}).  On the other hand, in Section 4 we show  every separable metrizable cut-point space densely embeds into a reducible metrizable continuum with no cut points (Theorem \ref{54}). An obvious example is the one-point compactification of $\mathbb R$.   
  A locally connected fan of long lines shows this type of embedding is not possible for all  Tychonoff cut-point spaces (see Example \ref{ex2}).  
 \begin{figure}[h]
  \centering
   \includegraphics[scale=.5]{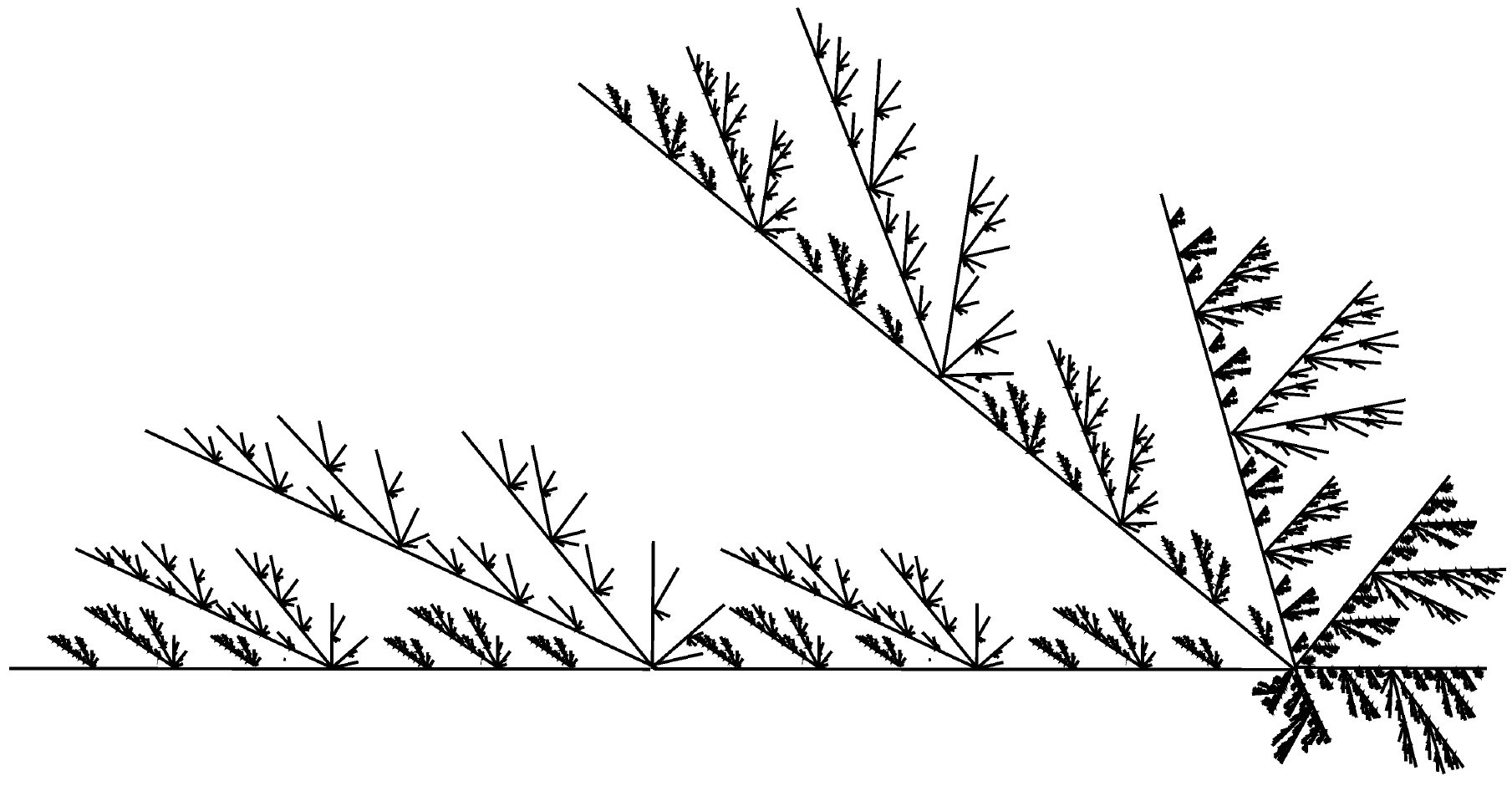}
   \caption{Universal plane dendrite}
 \end{figure} 
\subsection{Terminology} 

A  \textit{continuum} is a connected compact Hausdorff space.   An \textit{arc} is a continuum homeomorphic to the interval $[0,1]$.  A \textit{Hausdorff arc} is a linearly ordered (Hausdorff) continuum.

A  connected space $X$ is \textit{dendritic} if every two points are separated by some other point.  
Two points $a$ and $b$ are \textit{separated by} a third point $c$ if $X\setminus \{c\}$ is the union of two disjoint open sets, one containing $a$ and the other containing $b$. A \textit{dendron} is a  dendritic compact Hausdorff space, and a \textit{dendrite} is a metrizable dendron.

A topological space $X$  is \textit{connected im-kleinen} at $x\in X$ provided $x$ has  arbitrarily small connected neighborhoods. If $X$ is connected im-kleinen at each of its points, then  $X$ is  \textit{locally connected}.

A continuum $X$ \textit{indecomposable} if every proper subcontinuum of $X$ is nowhere dense. A continuum  $X$ is \textit{reducible} if for every two points $a,b\in X$ there exists a proper subcontinuum of $X$ containing $a$ and $b$.   If no proper subcontinuum of $X$ contains both $a$ and $b$, then $X$ is \textit{irreducible between $a$ and $b$}.  A continuum which is irreducible between some two of its points is said to be \textit{irreducible}.

A \textit{compactification} of a Tychonoff space $X$ is a compact Hausdorff space which has a dense subspace homeomorphic to $X$. $\beta X$ denotes the \textit{Stone-\v{C}ech compactification} of $X$. 

Each locally compact Hausdorff space $X$ has a compactification  $\gamma X$ such that the remainder $\gamma X\setminus X$ is zero-dimensional,  and disjoint closed subsets of $X$ with compact boundaries have disjoint closures in $\gamma X$. The canonical compactification of $X$ with these properties is called the \textit{Freudenthal compactification} of $X$.

\section{Proof of Theorem 1.1}

Let $X$ be a connected separable locally compact Hausdorff space. Let $\nc(X)$ denote the set of non-cut points of $X$.  Suppose $|\nc(X)|<\mathfrak c$.

Let $D$ be the Freudenthal compactification of $X$.

\begin{ucl}\label{35}Every point in $X$ is a cut point of $D$.\end{ucl}

\begin{proof} Let $x\in X$, and write $X\setminus \{x\}=U\sqcup V$. Let $W$ be an open subset of $X$ such that $x\in W$ and $\overline W$ is compact. Then $U\setminus W$ and $V\setminus W$ are disjoint closed subsets of $X$ with compact boundaries.  Thus  $\overline{U\setminus W}\cap \overline{V\setminus W}=\varnothing$.  It follows that $D\setminus \{x\}$ is the union of the two disjoint open sets $(U\cap W)\cup \overline{U\setminus W}$ and $(V\cap W)\cup \overline{V\setminus W}$.\end{proof}

\begin{ucl}\label{32}For every two non-degenerate subcontinua $K,L\subseteq D$, if $K\subseteq L$ then $K$ contains a cut point of $L$.  \end{ucl}

\begin{proof}Suppose $K$ and $L$ are non-degenerate subcontinua of $D$ and $K\subseteq L$.   Since $D\setminus X$ is zero-dimensional and compact, $K\cap X$ is a non-empty open subset of $K$.   Every open subset of a continuum has cardinality at least $\mathfrak c$.  Hence  $|\nc(X)|<\mathfrak c$ implies $K$ contains uncountably many cut points of $X$.  And by Claim \ref{35}, for each $x\in K\cap X\setminus \nc(X)$ we can write $D\setminus \{x\}=U_x\sqcup V_x$.

For a contradiction, suppose $L\setminus \{x\}$ is connected for all $x\in K\cap X\setminus \nc(X)$. Then we  may assume $L\setminus \{x\}\subseteq U_x$.   For any two points $x\neq y\in K\cap X\setminus \nc(X)$ we have $V_x\subseteq U_y\cup V_y$ and $x\in U_y$.  The set $V_x\cup \{x\}$ is connected, therefore $V_x\cup \{x\}\subseteq U_y$ and $V_x\cap V_y=\varnothing$. Thus $\{V_x:x\in K\cap X\setminus \nc(X)\}$ is an uncountable collection of pairwise disjoint non-empty open subsets of $D$.  This contradicts the fact that $D$ is separable.   Therefore $K$ contains a cut point of $L$.\end{proof}

By Claim \ref{32} and \cite[Theorem 1]{warrr}, $D$ is a dendron. Separable dendrons are metrizable by \cite[Theorem I.5]{eb}. Thus $D$ is a dendrite. Clearly $D\setminus X$ contains no cut point of $D$, so by Claim \ref{35} $X$ is equal to the set of cut points of $D$.  This concludes our proof of Theorem \ref{t2}(i). 


Toward proving Theorem \ref{t2}(ii), note that the set of cut points of any dendrite is a countable union of arcs.  So by  part (i), $\nc(X)$ is a $G_\delta$-subset of $X$. Hence $|\nc(X)|<\mathfrak c$ implies $X$ is scattered (and countable). Every open subset of $X$ is perfect, so $\nc(X)$ is nowhere dense. This concludes the proof of Theorem \ref{t2}(ii).

\begin{uc}\label{pp}Every separable Hausdorff continuum with only countably many non-cut points is a dendrite.\end{uc} 


\section{Weakly ordered Tychonoff spaces}

In this section we show connected weakly ordered Tychonoff spaces are precisely those cut-point spaces which can be densely embedded into irreducible continua.  These include graphs of certain functions defined on the real line.  For a non-trivial example, let $\varphi(t)=\sin(1/t)$ for $t\in \mathbb R\setminus \{0\}$ and put $\varphi(0)=0$. Now let $\mathbb Q =\{q_n:n<\omega\}$ be an enumeration of the rationals, and define $f:\mathbb R \to [0,1]$  by $f(t)=\sum_{n=1}^\infty \varphi(t-q_n)\cdot 2^{-n}.$  The graph $X:=\{\langle t,f(t)\rangle:t\in \mathbb R\}$  is connected, and the elements of $X$ are ordered by the first coordinate projection. This example is due to Kuratowski and Sierp\-i\'n\-ski \cite{ks}. More generally, for every $n\leq \omega$  Duda \cite[Theorem 6]{dud}  constructed a  function $f:\mathbb R\to [0,1]^n$  whose graph is  $n$-dimensional and connected. 


To prove the first two results, we need   the following fact: 
\renewenvironment{quote}{%
   \list{}{%
     \leftmargin1cm   
     \rightmargin\leftmargin
   }
   \item\relax
}
{\endlist}
\begin{quote}
If  $U$ and $V$ are disjoint open subsets of a Tychonoff space $X$, and $W$ is an open subset of $\beta X$ such that $W\cap X=U\cup V$, then the sets $W\cap \cl_{\beta X}U$ and $W\cap \cl_{\beta X} V$ are disjoint $\beta X$-open sets unioning to $W$.  
\end{quote}
Proofs may be found in the proofs of \cite[Lemma 1.4]{walk} and  \cite[Theorem 4]{lip}.

\begin{ut}\label{41}If $X$ is a connected Tychonoff space, then every cut point of $X$ is a cut point of $\beta X$.\end{ut}

\begin{proof}Suppose $x\in X$ is a cut point. Write $X\setminus \{x\}=U\sqcup V$.  Let $W=\beta X\setminus \{x\}$. By the fact above, $\beta X\setminus \{x\}$ is the union of two disjoint open sets $[\cl_{\beta X} U]\setminus \{x\}$ and $[\cl_{\beta X} V]\setminus \{x\}$. \end{proof}

\begin{ut}\label{42}If $X$ is a connected weakly orderable Tychonoff space, then $\beta X$ is an irreducible continuum.\end{ut}

\begin{proof}Let $X$ be a connected weakly ordered Tychonoff space.  Let $Y$ be a Hausdorff arc compactification of $X$ in the weak order topology. Let $y_0$ and $y_1$ be the endpoints of $Y$, and note that  $Y\setminus X\subseteq \{y_0,y_1\}$.  Let $f:X\hookrightarrow Y$ be the identity, and let $\beta f:\beta X\to Y$ be the Stone-\v{C}ech extension of $f$. Then there exist $p\in \beta f^{-1}\{y_0\}$ and $q\in \beta f^{-1}\{y_1\}$. We claim $\beta X$ is irreducible between $p$ and $q$.  

Let $K$ be any subcontinuum of $\beta X$ containing $p$ and $q$.  We show $X\setminus \{y_0,y_1\} \subseteq K$.  Let $x\in (y_0,y_1)$.  Take $U=f^{-1}[y_0,x)$ and $V=f^{-1}(x,y_1]$ and $W=\beta X\setminus \{x\}$.  By the fact above, $[\cl_{\beta X}U]\setminus \{x\}$ and $[\cl_{\beta X} V]\setminus \{x\}$ are disjoint $\beta X$-open sets covering $\beta X\setminus \{x\}$.   Since $p\in \cl_{\beta X}U$, $q\in \cl_{\beta X}V$, and $K$ is connected, we have $x\in K$. Thus $X\setminus \{y_0,y_1\}\subseteq K$.  So $K$ contains a dense subset of $\beta X$, therefore $K=\beta X$. 
\end{proof}

 \begin{ul}\label{43}Let $X$ be a cut point space.  For all $x_0,x_1\in X$ there are three disjoint  non-empty open sets $U$, $W$ and $V$ such that $X\setminus \{x_0,x_1\}=U\cup W\cup V$.
\end{ul}

\begin{proof}Write $X\setminus \{x_0\}=U\sqcup W_0$ so that $x_1\in W_0$.  Write $X\setminus \{x_1\}=W_1\sqcup V$ with $x_0\in W_1$.  Let $W=W_0\cap W_1$. Note that $$X\setminus W=X\setminus (W_0\cap W_1)=(X\setminus W_0)\cup (X\setminus W_1)\subseteq U\cup V\cup \{x_0,x_1\}.$$ So $X\setminus \{x_0,x_1\}=U\cup W\cup V$. Clearly $U\cap W=\varnothing$  and $V\cap W=\varnothing$.  Finally, $U\cup \{x_0\}$ is connected, so $U\cup \{x_0\}\subseteq W_1$. Therefore $U\cap V=\varnothing$.
\end{proof}
 
 \begin{ut}\label{44}Let $X$ be a Tychonoff cut-point space, and suppose $\beta X$ is an  irreducible continuum.  Then $X$ is weakly ordered. \end{ut}
 
\begin{proof}Let $p,q\in \beta X$ such that $\beta X$ is irreducible between $p$ and $q$. 

 We claim that every indecomposable subcontinuum of $\beta X$ is nowhere dense.  Suppose to the contrary  that $I$ is an indecomposable subcontinuum of $\beta X$, and $I$ contains a non-empty $\beta X$-open subset $G$.  Let $x_0,x_1\in G\cap X$ and write $X\setminus \{x_0,x_1\}=U\sqcup W\sqcup V$ as in Lemma \ref{43}. Then  $U\cap G$ and $V\cap G$ are non-empty open sets.  Each composant of $I$ is dense in $I$, and every proper subcontinuum of $I$ is nowhere dense.  So there is a nowhere dense subcontinuum $N\subseteq I$ which intersects both $\cl_{\beta X}(U\cap G)$ and $\cl_{\beta X}(V\cap G)$. Since $U\cup \{x_0\}$ and $\{x_1\}\cup V$ are connected, we find that $K:=\cl_{\beta X}U\cup N\cup \cl_{\beta X} V$ is a  proper subcontinuum of $\beta X$. By irreducibility  between $p$ and $q$, $\{p,q\}\not\subseteq K$. Without loss of generality, assume $p\notin K$.  Then $p\in \cl_{\beta X} W\subseteq \cl_{\beta X}(X\setminus U)\cap \cl_{\beta X}(X\setminus V)$.  Note that $X\setminus U$ and $X\setminus V$ are connected, and $q\in \cl_{\beta X}(X\setminus U)\cup\cl_{\beta X}(X\setminus V)$.  Therefore  $p$  and $q$ are contained proper subcontinuum of $Y$. This is a contradiction.

By Gordh \cite{gor} and the claim above, $\beta X$ is a generalized $\lambda$-type continuum. That is, there is a Hausdorff arc $Y$ and a mapping $\lambda:\beta X\to Y$  such that $\{\lambda^{-1}\{y\}:y\in Y\}$ is an upper semi-continuous decomposition of $\beta X$ into maximal nowhere dense subcontinua. 

To prove $X$ is weakly ordered, it suffices to show $\lambda\restriction X$ is one-to-one. Suppose $x_0,x_1\in X$ and $\lambda(x_0)=y= \lambda(x_1)$.  If $x_0\neq x_1$ then we may write  $X\setminus \{x_0,x_1\}=U\sqcup W\sqcup V$ as in Lemma \ref{43}.  Then $K:=\cl_{\beta X} [U\cup \lambda^{-1}\{y\}\cup V]$ is a subcontinuum of $\beta X$ which contains both $p$ and $q$.  Since $\lambda^{-1}\{y\}$ is nowhere dense,  $K$ is a proper subset of $\beta X$. This  violates irreducibility between $p$ and $q$. Therefore $x_0=x_1$ and $\lambda$ is one-to-one. \end{proof}

\begin{ur}We observe that $\lambda^{-1}\{\lambda(x)\}$ is the union of two continua $H$ and $K$ such that $H\cap K=\{x\}$, and $\lambda^{-1}\{\lambda(x)\}=\{x\}$ if and only if $X$ is connected im-kleinen at $x$.\end{ur}

\begin{uc}\label{46}A Tychonoff cut-point space $X$ is weakly orderable  if and only if $\beta X$ is an irreducible continuum. \end{uc}

\begin{proof}Combine Theorems \ref{42} and \ref{44}. \end{proof}

\begin{ut}\label{47}Let $X$ be a   connected weakly orderable normal space.  Then   densely embeds into an irreducible continuum of the same weight as $X$. In particular, if $X$ is separable metrizable then $X$   densely embeds into an  irreducible metrizable continuum.\end{ut}

\begin{proof} Let $X$ be a connected weakly orderable normal space. Let $\kappa$ be weight of $X$, i.e. the least cardinality of a basis for $X$. 

By Theorem \ref{42}, $\beta X$ is irreducible between two points $p$ and $q$.  Let $\{U_\alpha:\alpha<\kappa\}$ be a basis for $X\setminus \{p,q\}$ with each $U_\alpha\neq\varnothing$. For each $\alpha<\kappa$ we have that $\beta X\setminus U_\alpha$ is the union of two disjoint compact sets $A_\alpha$ and $B_\alpha$ with $p\in A_\alpha$ and $q\in B_\alpha$.  

By Urysohn's Lemma, for every $\alpha<\kappa$ there is a continuous function $f_\alpha:X\to[0,1]$ such that $f_\alpha[A_\alpha\cap X]=0$ and $f_\alpha[B_\alpha\cap X]=1$.  Define $f:X\to [0,1]^\kappa$ by $f(x)=\langle f_\alpha(x)\rangle_{\alpha<\kappa}$. 

Let $g:X\to [0,1]$ be a homeomorphic embedding of $X$ into the Tychonoff cube $[0,1]^\kappa$, and put  $h=f\times g$. Then $h:X\to [0,1]^{ \kappa}\times [0,1]^{ \kappa}$ is a homeomorphism, and $\overline{h[X]}$ is a continuum irreducible between $\beta h(p)$ and $\beta h(q)$.   Here, $\beta h:\beta X\to  \overline{h[X]}$ is the Stone-\v{C}ech extension of $h$.  \end{proof}


\section{Non-cut points in compactifications}

The non-cut point existence theorem for connected  compact spaces, stated below,  was originally proved by R.L. Moore \cite{more} in the context of metric spaces. It was generalized for T$_1$ spaces by G.T. Whyburn \cite{why}, and finally  for all topological spaces by B. Honari and Y. Bahrampour  in \cite{poo}. 

\begin{ut}[Theorem 3.9 in \cite{poo}]If $X$ is a compact connected topological space with more than
one point, then X has at least two non-cut points. \end{ut}

No separation axioms are needed to prove the next four results.

\begin{ut}\label{52}If $X$ is a cut-point space, then for every $x\in X$ and connected component $C$ of $X\setminus \{x\}$, $C\cup \{x\}$ is non-compact.\end{ut}  

\begin{proof}Let $X$ be a cut-point space.  Let $x\in X$, and let $C$ be a connected component of $X\setminus \{x\}$. Suppose $C\cup \{x\}$ is compact.  We will reach a contradiction by finding a non-cut point of $X$ in $C$.  

Observe that $X\setminus C$ is connected. For if $X\setminus C$ is the union of two nonempty and disjoint separated sets $A$ and $B$ with $x\in A$, then $C\cup B$ is a connected subset of $X\setminus \{x\}$ bigger than $C$.  Also, $C$ is closed in the subspace $X\setminus \{x\}$, implying $\overline{C}\in \{C,C\cup \{x\}\}$.

\textit{Case 1}: $\overline{C}=C\cup \{x\}$. Then  $C\cup \{x\}$ is a compact connected set with more than one point and thus has a non-cut point $y\in C$.  Observe that  $X\setminus \{y\}$ is equal to the union of the two connected sets $(C\cup \{x\})\setminus \{y\}$ and $X\setminus C$ which have the point $x$ in common. Therefore $y$ is a non-cut point of $X$. 

\textit{Case 2}: $\overline{C}=C$. Then $C$ is compact and connected. Additionally,  $X\setminus C$ is connected implies $C$ has more than one point. Thus $C$ has two non-cut points $y_0$ and $y_1$.   There exists $b\in 2$ such that $\{y_b\}\neq \overline{X\setminus C}\cap C$.  By connectedness of $X$ we have $\overline{X\setminus C}\cap C =\overline{X\setminus C}\cap \overline{C}\neq\varnothing$. By  the choice of $b$ it follows that $\overline{X\setminus C}\cap (C\setminus \{y_b\})\neq\varnothing$. Thus $X\setminus \{y_b\}$ is the union of two non-separated connected sets   $X\setminus C$ and $C\setminus \{y_b\}$. Therefore $X\setminus \{y_b\}$ is connected and $y_b$ is a non-cut point of $X$. 

In each case we reached a contradiction. Therefore $C\cup \{x\}$ is non-compact. \end{proof}


\begin{uc}\label{53}Let $X$ be a locally connected cut-point space. If $x\in X$ has a compact neighborhood, and $\{x\}$ is closed, then $X\setminus \{x\}$ has only finitely many connected components.  \end{uc}

\begin{proof}Suppose $N$ is a compact neighborhood of $x$, and $\{x\}$ is closed.  Let $\{C_\alpha:\alpha<\kappa\}$ be the set of connected components of $X\setminus \{x\}$. Since $X$ is locally connected and $X\setminus \{x\}$ is open,  each $C_\alpha$ is open.  

By Theorem \ref{52} and the fact that $\{x\}\cup C_\alpha$ is closed, we have $C_\alpha\setminus N\neq\varnothing$ for each $\alpha<\kappa$.  Since $C_\alpha$ is a relatively clopen subset of $X\setminus \{x\}$, by connectedness of $X$ we  have $x\in \overline{C_\alpha}$.  So $C_\alpha\cap \partial N\neq\varnothing$.  

The $C_\alpha$'s are pairwise disjoint, so no proper subcollection of $\{C_\alpha:\alpha<\kappa\}$ covers $\partial N$. A finite subcollection covers $\partial N$  by compactness, so $\kappa$ is finite. \end{proof}

\begin{uc}\label{53}Let $X$ be a cut-point space which is a dense subset of a compact space $Y$.  If $Y\setminus X$ is connected, then $Y$ has no cut points.\end{uc}

\begin{proof}Suppose $Y\setminus X$ is connected.  For each   $p\in Y\setminus X$, the set $Y\setminus \{p\}$ is connected because it has dense connected subset $X$.  Now let $x\in X$. By Theorem \ref{52},  $(\cl_{Y}C)\setminus X\neq\varnothing$ for  each connected component $C$ of $X\setminus \{x\}$.  Since $Y\setminus X$ is connected, this implies $Y\setminus \{x\}$ is connected. \end{proof}

\begin{uc}The one-point compactification of a locally compact cut-point space has no cut points. \end{uc}

To prove the next theorem, we use the fact that every connected separable metrizable space has a metrizable compactification with path-connected remainder.  This was proved by Jan J. Dijkstra in \cite{po}. 

\begin{ut}\label{54}Every separable metrizable cut-point space densely embeds into a reducible metrizable continuum with no cut points.\end{ut}

\begin{proof}Let $X$ be a separable metrizable cut-point space.  By  \cite[Theorem 1]{po},  there is a metrizable compactification $\gamma X$  such that $\gamma X\setminus X$ is path-connected.  By Corollary \ref{53}, $\gamma X$ has no cut points. 

It remains to show $\gamma X$ is reducible.  To that end, let $p,q\in \gamma X$. We will assume $p\neq q$, and exhibit a proper subcontinuum of $\gamma X$ containing $p$ and $q$.  If $p,q\in \gamma X\setminus X$, then there is an arc $A\subseteq \gamma X\setminus X$ with $p,q\in A$. 

Now suppose $p\in X$ or $q\in X$.  Assume $p\in X$, and write $X\setminus \{p\}=U\sqcup V$.  Without loss of generality, $q\in \cl_{\gamma X} V$.  Then $\cl_{\gamma X} (\{p\}\cup V)$ is a proper subcontinuum of $Y$ containing $p$ and $q$.  \end{proof}


The following example shows Theorem \ref{54} does not generalize to Tychonoff spaces.

\begin{ue}\label{ex2}Let $[0,\omega_1)$ denote the $\omega_1$-long line, which is defined as $\omega_1 \times [0,1)$ in the lexicographic order topology.  Endow $A:=[0,\omega_1)\times (\{0\}\cup \{1/n:n=1,2,3,...\})$ with  the product topology. Then the  locally connected fan  $$X:=A/\{\langle x,y\rangle\in A:x=0\text{ or }y=0\}$$ is a Tychonoff cut-point space.  

Define $\overline{X}$  similarly, with  $B:=[0,\omega_1]\times (\{0\}\cup \{1/n:n=1,2,3,...\})$ in the place of $A$.  Here $[0,\omega_1]=[0,\omega_1)\cup \{\omega_1\}$   denotes the one-point compactification of $[0,\omega_1)$. 

If $f$ is any continuous real-valued function on $X$, then $f\restriction [0,\omega_1)\times \{1/n\}$ is eventually constant.  We observe that  $f$ continuously extends $\overline X $ by  mapping $\langle \omega_1,1/n\rangle$ to the eventually constant value of $f\restriction [0,\omega_1)\times \{1/n\}$.  So $\overline X =\beta X$. Thus if $\gamma X$  any compactification of $X$, then there is a continuous surjection $\beta \iota:\overline X \to \gamma X$ extending identity $\iota:X\to X$.  The function $\beta \iota$ is finite-to-one, and $\gamma X \simeq \{\beta \iota^{-1}\{p\}:p\in \gamma X \}$ in the quotient topology. Thus, $\gamma  X$ is obtained from $\overline X $ by collapsing finite subsets of $\{\langle \omega_1,1/n\rangle:n=1,2,3,...\}$. We see now  that $\gamma  X\setminus  X\simeq \omega$ and $\gamma  X\setminus \{\langle 0,0\rangle\}$ has infinitely many connected components.  In particular,  $\langle 0,0\rangle$ is a cut point of $\gamma X$.  \end{ue}


We say that a cut point $x\in X$  is  \textit{persistent} if $x$  is a cut point of every compactification of $X$.  In Example \ref{ex2}, $\langle 0,0\rangle$ is a persistent cut point of $X$. All other cut points of $X$ are non-persistent. To see this, take  $\overline X $ and for each $n=1,2,3,...$ glue together  the two points $\langle \omega_1,1/(2n-1)\rangle$ and $\langle \omega_1,1/(2n)\rangle$. The resulting continuum has only one cut-point: $\langle 0,0\rangle$.

\begin{ut}\label{58}Let $X$ be a locally connected Tychonoff cut-point space.   If $x\in X$ has a compact neighborhood, then $x$ is non-persistent. \end{ut}

\begin{proof}
By Corollary \ref{53}, $X\setminus \{x\}$ has only finitely many components $C_0,C_1,...,C_{n-1}$.  By Theorem \ref{52}, for each $i<n$ there exists $p_i\in [\cl_{\beta X}C_i]\setminus X$. The quotient $\beta X/\{p_i:i<n\}$ is a compactification of $X$ in which $x$ is a non-cut point. \end{proof}

\begin{uq}Does every Tychonoff cut-point space have a non-persistent cut point?\end{uq} 

A positive answer to Question 1 could be viewed as a generalization of the non-cut point existence theorem for Hausdorff continua, since each cut point of a continuum is persistent. 

\end{document}